\documentclass{amsart}

\usepackage{amssymb}

\newcommand{\Q}{\mathbb{Q}}
\newcommand{\C}{\mathbb{C}}

\newcommand{\R}{\mathbb{R}}
\newcommand{\N}{\mathbb{N}}

\newcommand{\cat}{^\frown}
\newcommand{\dom}{\operatorname{dom}}
\newcommand{\ran}{\operatorname{ran}}

\newcommand{\norm}[1]{\left\| #1 \right\|}
\newcommand{\supp}{\operatorname{supp}}

\theoremstyle{theorem}
\newtheorem{theorem}{Theorem}[section]
\newtheorem{lemma}[theorem]{Lemma}

\theoremstyle{definition}
\newtheorem{definition}[theorem]{Definition}

\theoremstyle{theorem}

\theoremstyle{theorem}
\newtheorem{proposition}[theorem]{Proposition}

\theoremstyle{theorem}

\theoremstyle{theorem}

\theoremstyle{definition}

\theoremstyle{theorem}

\numberwithin{equation}{section}

\begin{document}
\title{Computable copies of $\ell^p$}
\author{Timothy H. McNicholl}
\address{Department of Mathematics\\
Iowa State University\\
Ames, Iowa 50011}
\email{mcnichol@iastate.edu}
\thanks{Subsection \ref{subsec:proof.thm.main.2} previously appeared in the conference proceedings of CiE 2015 \cite{McNicholl.2015}.  The author's participation in CiE 2015 was supported by a Simons Foundation Collaboration Grant for Math\-e\-ma\-ti\-cians}

\begin{abstract}
Suppose $p$ is a computable real so that $p \geq 1$.  It is shown that the halting set can compute a surjective linear isometry between any two computable copies of $\ell^p$.  It is also shown that this result is optimal in that when $p \neq 2$ there are two computable copies of $\ell^p$ with the property that any oracle that computes a linear isometry of one onto the other must also compute the halting set.  Thus, $\ell^p$ is $\Delta_2^0$-categorical and is computably categorical if and only if $p = 2$.  It is also demonstrated that there is a computably categorical Banach space that is not a Hilbert space.  These results hold in both the real and complex case.
\end{abstract}
\maketitle

\section{Introduction}\label{sec:intro}

We start by considering the very general question ``Given two computable and linearly isometric Banach spaces, how hard is it to compute a linear isometry from one onto the other?"  (Roughly speaking, a Banach space is computable if there are algorithms that compute its scalar multiplication, vector addition, and norm.)  We specialize this question to the case of Banach spaces that are linearly isometric to $\ell^p$ where $p \geq 1$ is a computable real (i.e. a real whose decimal expansion is computable).  Our first result is that this is no harder than computing membership in the halting set.   Namely, we show that when $p$ is a computable real so that $p \geq 1$, the halting set is capable of computing a surjective linear isometry between any two computable copies of $\ell^p$.  Our second result is that this problem is not easier than the halting set.  Namely,
when $p$ is a computable real so that $p \geq 1$ and $p \neq 2$, there are two computable copies of $\ell^p$ so that any oracle that computes a surjective linear isometry from one onto the other must also compute the halting set.   It is already known that any two computable copies of $\ell^2$ are computably linearly isometric \cite{Pour-El.Richards.1989}.  This is essentially due to the fact that $\ell^2$ is a Hilbert space and mirrors the classical fact that any two infinite-dimensional separable Hilbert spaces are linearly isometric \cite{Halmos.1998}.

The first of our two results is based on a sharpening of an inequality due to J. Lamperti which we prove in Section \ref{sec:classical}. 
In the main, our second result was previously shown for $p = 1$ by Pour-El and Richards \cite{Pour-El.Richards.1989}.  Their proof rests on an observation about the extreme points of the closed unit ball of $\ell^1$ that does not generalize to $\ell^p$ when $p > 1$.  The proof presented here uses the characterization of the linear isometries of $L^p$ spaces due to S. Banach and J. Lamperti \cite{Banach.1987}, \cite{Fleming.Jamison.2003}, \cite{Lamperti.1958}.  

Our findings can be recast in the setting of computable categoricity.  A mathematical structure is \emph{computably categorical} if any two of its computable copies are isomorphic via a computable map.  A structure is \emph{$\Delta^0_2$-categorical} if the halting set can compute an isomorphism between any two of its computable copies \cite{Ash.Knight.2000}, \cite{Fokina.Harizanov.Melnikov.2014}.  Our results can be interpreted in the setting of computable categoricity by replacing `isomorphism' with `surjective linear isometry'; i.e. when $p$ is a computable real so that $p \geq 1$, $\ell^p$ is $\Delta_2^0$-categorical, and $\ell^p$ is computably categorical if and only if $p = 2$.   The latter resolves a question posed by A.G. Melnikov in 2013 \cite{Melnikov.2013}.

Although our theorems are proven for the complex version of $\ell^p$, they also hold for the real version of $\ell^p$.

The paper is organized as follows.  Section \ref{sec:background} covers background and preliminaries from functional analysis and computable analysis.  Section \ref{sec:overview} gives an overview of the proof that $\ell^p$ is $\Delta_2^0$-categorical.  The remainder of the work is then divided into two parts each corresponding to a different mathematical universe: the classical world (Section \ref{sec:classical}), wherein we have full access to all the concepts, principles, and methods of classical mathematics, and the computable world (Section \ref{sec:computable}) wherein we can only see approximations of classical objects and can only access computable operations on these approximations.  In Section \ref{sec:additional}, we use the methods developed in the previous sections to provide simple proofs that there is a computably categorical Banach space that is not a Hilbert space and that $\ell^p$ has a computable copy if and only if $p$ is computable.  Section \ref{sec:conclusion} presents concluding remarks. 

\section{Background and preliminaries}\label{sec:background}

\subsection{Background and preliminaries from functional analysis}\label{subsec:background.fa}

Throughout this paper, it is assumed that all Banach spaces are Banach spaces over the field of complex numbers $\C$.

We begin with some notation and terminology.  Let $\mathcal{B} = (V, \cdot, +, \norm{\ })$ be a Banach space.  By a \emph{subspace} of $\mathcal{B}$ we will always mean a linear subspace of $\mathcal{B}$ that is topologically closed.  When $S \subseteq V$ and $F \subseteq \C$, we let $\mathcal{L}_F(S)$ denote the set of all linear combinations of vectors in $S$ whose coefficients lie in $F$; i.e.
\[
\mathcal{L}_F(S) = \left\{ \sum_{j = 0}^M \alpha_j v_j\ :\ M \in \N\ \wedge\ \alpha_0, \ldots, \alpha_M \in F\ \wedge\ v_0, \ldots, v_M \in S\right\}.
\]
The \emph{subspace generated by $S$} is the closure of the linear span of $S$; we denote this by $\langle S \rangle$.  
We say that $G \subseteq V$ is a \emph{generating set} for $\mathcal{B}$ if 
it generates all of $\mathcal{B}$; i.e. $V = \langle G \rangle$.  For example, let $e_n = \chi_{\{n\}}$ for all $n \in \N$ (where $\chi_A$ denotes the characteristic function of $A$).
Then, $E := \{e_0, e_1, \ldots\}$ is a generating set for $\ell^p$ which we refer to as the \emph{standard generating set} for $\ell^p$.  Also, the set of all $f_n(x) = x^n$ for $n \in \N$ is a generating set for $C[0,1]$.

A map between two Banach spaces is \emph{linear} if it preserves scalar multiplication and vector addition; it is an \emph{isometry} (or is \emph{isometric}) if it preserves the metric induced by the norm; i.e. $\norm{T(x) - T(y)} = \norm{x - y}$.  Thus, every isometry is injective.  An \emph{endomorphism} of a Banach space is a linear (but not necessarily isometric) map of the space into itself.

When $p$ is a positive number, $\ell^p$ denotes the space of all functions $f : \N \rightarrow \C$ so that 
\[
\sum_{n = 0}^\infty |f(n)|^p < \infty.
\]
$\ell^p$ is a vector space over $\C$ with the usual scalar multiplication and vector addition.  When $p \geq 1$ it is a Banach space under the norm defined by 
\[
\norm{f}_p = \left( \sum_{n = 0}^\infty |f(n)|^p \right)^{1/p}.
\]  
It is often convenient to view $\ell^p$ as $L^p(\mu)$ where $\mu$ is the counting measure on $\N$.

When $f \in \ell^p$, the \emph{support} of $f$ is the set of all $t \in \N$ so that $f(t) \neq 0$; we denote this set by $\supp(f)$.  If $f_0, f_1, \ldots$ are vectors in $\ell^p$ so that $\supp(f_m) \cap \supp(f_n) = \emptyset$ whenever $m \neq n$, then we say that $f_0, f_1, \ldots$ are \emph{disjointly supported}.  Note that if $f,g \in \ell^p$ are disjointly 
supported then $\norm{f +g }_p^p = \norm{f}_p^p + \norm{g}_p^p$.

We now describe a simple numerical test for disjointness of support when $p \neq 2$.  
When $z, w \in \C$ let:
\[
\sigma_1(z, w) = |2 (|z|^p + |w|^p) - (|z - w|^p + |z + w|^p)|
\]
In 1958, J. Lamperti proved that $\sigma_1(z,w) = 0$ iff $zw = 0$ and that the sign of $2 (|z|^p + |w|^p) - (|z - w|^p + |z + w|^p)$ depends only on $p$.  Define $\sigma_1(f,g)$ to be $\sum_n \sigma_1(f(n), g(n))$.  Thus, 
$\sigma_1(f,g) = |2(\norm{f}^p + \norm{g}^p) - (\norm{f - g}^p + \norm{f + g}^p)|$ and $\sigma_1(f,g) = 0$ if and only if $f,g$ are disjointly supported.  Note also that $\sigma_1$ is invariant under linear isometries.  Thus, 
every isometric endomorphism of $\ell^p$ preserves the `disjoint support' relation.   That is, if $T : \ell^p \rightarrow \ell^p$ is a linear isometry, then $T(f)$ and $T(g)$ are disjointly supported whenever $f,g \in \ell^p$ are disjointly supported.  

When $f,g \in \ell^p$, write $f \preceq g$ if $f(n) = 0$ whenever 
$g(n) \neq f(n)$.  It follows that $\preceq$ is a partial order of $\ell^p$.  Note that the atoms of this partial order are the nonzero scalar multiples of the $e_n$'s.  Note also that $f \preceq g$ if and only if $g - f$ and $f$ are disjointly supported.  Thus, $\preceq$ is preserved by isometric endomorphisms of $\ell^p$.

The proof that $\ell^p$ is not computably categorical when $p \neq 2$ is based on the following.

\begin{theorem}[Banach-Lamperti]\label{thm:classification}
Suppose $1 \leq p < \infty$ and $p \neq 2$. Suppose $T$ is an endomorphism of $\ell^p$.  Then, $T$ is a surjective isometric endomorphism of $\ell^p$ if and only if there are unimodular constants $\lambda_0, \lambda_1, \ldots$ and a permutation of $\N$, $\phi$, so that $T(e_n) = \lambda_n e_{\phi(n)}$ for all $n$.
\end{theorem}

In his seminal text on linear operators, S. Banach stated Theorem \ref{thm:classification} for the case of $\ell^p$ spaces over the reals \cite{Banach.1987}.  He also stated a classification of the linear isometries of $L^p[0,1]$ in the real case.  Banach's proofs of these claims were sketchy and did not easily generalize to the complex case.  In 1958, J. Lamperti rigorously proved a generalization of  Banach's claims to real and complex $L^p$-spaces of $\sigma$-finite measures \cite{Lamperti.1958}.  Theorem \ref{thm:classification} follows from J.  Lamperti's work as it appears in Theorem 3.2.5 of \cite{Fleming.Jamison.2003}.   Note that Theorem \ref{thm:classification} fails when $p = 2$.  For, $\ell^2$ is a Hilbert space.  So, if $\{f_0, f_1, \ldots\}$ is any orthonormal basis for $\ell^2$, then there is a unique surjective linear isometry $T$ of $\ell^2$ so that $T(e_n) = f_n$ for all $n$.

\subsection{Background and preliminaries from computable analysis}\label{subsec:background.ca}

We assume the reader is familiar with the basic notation and terminology of computability theory as expounded in \cite{Cooper.2004}.  We cover here the basic notions from computable analysis necessary to understand the results herein.  A more expansive treatment can be found in  \cite{Pour-El.Richards.1989}, \cite{Weihrauch.2000}.  

Suppose $\mathcal{B}$ is a Banach space and $F = \{f_0, f_1, \ldots \} \subseteq \mathcal{B}$ is a generating set for $\mathcal{B}$.  We say that $F$ is an \textit{effective generating set} for $\mathcal{B}$ if there is an algorithm that, given any $f \in \mathcal{L}_{\mathbb{Q}(i)}(F)$ and a nonnegative integer $k$ as input computes a rational number $q$ so that $| \|f\| - q| < 2^{-k}$; less formally, the map $f \in \mathcal{L}_{\mathbb{Q}(i)}(F) \mapsto \|f\|$ is computable.  For example, $\{1, i\}$ is an effective generating set for $\C$, and the standard generating set of $\ell^p$ is an effective generating set for $\ell^p$.  On the other hand, if $|\zeta| = 1$, then 
$\zeta E := \{\zeta e_0, \zeta e_1, \ldots \}$ is also an effective generating set for $\ell^p$, even if $\zeta$ is incomputable.  We designate $\{1,i\}$ and $E$ as the default effective generating sets for $\C$ and $\ell^p$ respectively; i.e. when discussing computability on these spaces without mention of an effective generating set it is implicit that we are using the default generating set.
 
Suppose $F$ is an effective generating set for a Banach space $\mathcal{B}$. We say that a vector $g \in \mathcal{B}$ is \textit{computable with respect to} $F$ if there is an algorithm that given any nonnegative integer $k$ as input computes a vector $f \in \mathcal{L}_{\mathbb{Q}(i)}(F)$ so that $\norm{g - f} < 2^{-k}$.  Thus a point $z \in \C$ is computable (with respect to the default generating set) if there is an algorithm that given any $k \in \mathbb{N}$ as input, produces a rational point $q$ so that $|q - z| < 2^{-k}$.  A vector $f \in \ell^p$ is computable (with respect to the default generating set $E$) if there is an algorithm that given any $n,k \in \N$ as input computes a rational point $q$ so that $|q - f(n)| < 2^{-k}$.  On the other hand, if $\zeta$ is an incomputable unimodular point, then only the zero vector is computable with respect to both $E$ and $\zeta E$.

Suppose $\mathcal{B}$ is a Banach space.  When $f \in \mathcal{B}$, and $r > 0$, let $B(f; r)$ denote the open ball with center $f$ and radius $r$.  Suppose $F$ is an effective generating set for $\mathcal{B}$.  When $f \in \mathcal{L}_{\mathbb{Q}(i)}(F)$ and $r$ is a positive rational number, we call $B(f; r)$ a \textit{rational ball} (with respect to $F$). 

Suppose that for each $j \in \{1,2\}$, $F_j$ is an effective generating set for $\mathcal{B}_j$.  We say that a map $T:  \mathcal{B}_1 \rightarrow B_2$ is \textit{computable with respect to $(F_1, F_2$)} if there is an algorithm $P$ that meets the following three criteria.
\begin{enumerate}
\item \textbf{Approximation:} Given as input a rational ball with respect to $F_1$, $P$ either does not halt or produces a rational ball with respect to $F_2$.
\item \textbf{Correctness:} If $B_2$ is the output of $P$ on input $B_1$, then $T(f) \in B_2$ whenever $f \in B_1$.
\item \textbf{Convergence:} If $V$ is a neighborhood of $T(f)$, and if $U$ is a neighborhood of $f$, then $f$ belongs to a rational ball $B_1 \subseteq U$ with the property that $P$ halts on $B_1$ and produces a rational ball that is included in $U$. 
\end{enumerate}

When we speak of an algorithm accepting a rational ball $B(\sum_{j = 0}^M \alpha_j f_j; r)$ as input, we mean that it accepts some representation of the ball such as a code of the sequence $(r, M, \alpha_0, \ldots, \alpha_M)$ and similarly when we say it produces a rational ball as output we mean that it produces codes of the center and radius.  

It is well-known that many familiar functions of a complex variable (such as $\sin$, $\exp$) are computable (with respect to the generating set $\{1,i\}$ used in the domain and range).  
Note also that when $|\zeta| = 1$ the `multiplication by $\zeta$' operator on $\ell^p$ is computable with respect to $E$ and $\zeta E$.

A \emph{computable Banach space} consists of a pair $(\mathcal{B}, F)$ where $\mathcal{B}$ is a Banach space and $F$ is an effective generating set for $\mathcal{B}$.  
Unless the effective generating set truly requires explicit mention, for the sake of economy of expression we will just refer to `the computable Banach space $\mathcal{B}$'.

If $(\mathcal{B}_1, F_1)$ and $(\mathcal{B}_2, F_2)$ are computable Banach spaces, then we say 
a map $T : \mathcal{B}_1 \rightarrow \mathcal{B}_2$ is computable if it is computable with respect to $(F_1, F_2)$.   It easily follows that if $T : \mathcal{B}_1 \rightarrow \mathcal{B}_2$ is a computable surjective linear isometry, then $T^{-1}$ is also computable.  

If $(\mathcal{B}_1, F_1)$ and $(\mathcal{B}_2, F_2)$ are computable Banach spaces, then 
$(\mathcal{B}_1 \times \mathcal{B}_2, F_1 \times F_2)$ is a computable Banach space.  
Thus, if $\mathcal{B}$ is a computable Banach space, then vector addition is a computable map from $\mathcal{B} \times \mathcal{B}$ onto $\mathcal{B}$ and scalar multiplication is a 
computable map from $\C \times \mathcal{B}$ onto $\mathcal{B}$.  In addition the norm of $\mathcal{B}$ is a computable map from $\mathcal{B}$ into $[0, \infty)$.

Suppose $\mathcal{B}$ is a computable Banach space.  
A closed set $C \subseteq \mathcal{B}$ is \emph{c.e. closed} if the set of all rational balls that contain a point of $C$ is c.e..  An open set $U \subseteq \mathcal{B}$ is \emph{c.e. open} if the set of all rational balls included in $U$ is c.e..  Suppose $\mathcal{B}'$ is a computable Banach space and $f : \mathcal{B} \rightarrow \mathcal{B}'$ is computable.  It is well-known that if $U \subseteq \mathcal{B}'$ is c.e. open, then $f^{-1}[U]$ is c.e. open.  

The following is `folklore'.

\begin{proposition}\label{prop:bounding.principle}
Suppose $\mathcal{B}$ is a computable Banach space and $f : \mathcal{B} \rightarrow \R$ is a computable 
function with the property that $f(v) \geq d(v, f^{-1}[\{0\}])$ for all $v\in \mathcal{B}$.  Then, $f^{-1}[\{0\}]$ is c.e. closed.  
\end{proposition}

The computability notions we have covered are all relativized in the usual way.
We now formally state our two main theorems.

\begin{theorem}\label{thm:main.1}
Suppose $p$ is a computable real so that $p \geq 1$.  Then, whenever $\mathcal{B}_0$ and $\mathcal{B}_1$ are computable Banach spaces that are linearly isometric to $\ell^p$, the halting set computes a surjective linear isometry of $\mathcal{B}_0$ onto $\mathcal{B}_1$. 
\end{theorem}

\begin{theorem}\label{thm:main.2}
Suppose $p$ is a computable real so that $p \geq 1$ and $p \neq 2$.  Suppose $C$ is a \emph{c.e.} set.  Then, 
there is a computable Banach space $\mathcal{B}$ so that $C$ computes a surjective linear isometry of $\ell^p$ onto $\mathcal{B}$ and so that any oracle that computes a surjective linear isometry of $\ell^p$ onto $\mathcal{B}$ also computes $C$.   
\end{theorem} 

So if we take $C$ to be the halting set in Theorem \ref{thm:main.2}, then it follows that the problem of computing a linear isometric map of one computable copy of $\ell^p$ onto another is at least as hard as computing membership in the halting set.  

We close this section by mentioning some related work.  A.G. Melnikov and K.M. Ng have investigated computable categoricity questions with regards to the space $C[0,1]$ of continuous functions on the unit interval with the supremum norm.  In particular, they have shown that $C[0,1]$ is not computably categorical as a metric space nor as a Banach space \cite{Melnikov.2013}, \cite{Melnikov.Ng.2014}.  The study of computable categoricity for countable structures goes back at least as far as the 1961 work of A.I. Malcev.  The text of Ash and Knight has a thorough discussion of the main early results of this line of inquiry \cite{Ash.Knight.2000}.  The survey by Fokina, Harizanov, and Melnikov contains a wealth of recent results on computable categoricity and other directions in the countable computable structures program \cite{Fokina.Harizanov.Melnikov.2014}.

\section{Overview}\label{sec:overview}

The proof of Theorem \ref{thm:main.2} is fairly straightforward.   
Here, we set forth the key concepts and supporting intermediate results for the proof of Theorem \ref{thm:main.1}.  

We first reduce the problem to the computation of surjective isometric endomorphisms.  Fix a computable real so that $p \geq 1$.  Suppose $(\mathcal{B}, G)$ is a computable Banach space, and suppose $T$ is a linear isometric mapping of $\mathcal{B}$ onto $\ell^p$.  Then, $T[G]$ is an effective generating set for $\ell^p$, and $T$ is computable with respect to $(G, T[G])$.  
Thus, since inverses of computable surjective linear isometries are computable, the study of computable Banach spaces that are linearly isometric to $\ell^p$ can be reduced to the study of computability notions on $\ell^p$ with respect to different generating sets.  In particular, to prove Theorem \ref{thm:main.1}, it suffices to show that 
whenever $F$ is an effective generating set for $\ell^p$, the halting set computes a surjective isometric endomorphism of $\ell^p$ with respect to $(E,F)$.  

Now, suppose $g_0, g_1, \ldots, $ are disjointly supported unit vectors in $\ell^p$.  Then, there is a unique 
linear isometric map $T : \ell^p \rightarrow \ell^p$ so that $T(e_n) = g_n$ for all $n$; if the $g_n$'s generate 
$\ell^p$, then $T$ is also surjective.  Furthermore, if an oracle $X$ computes $\{g_n\}_{n = 0}^\infty$ with respect to an effective generating set $F$, then $X$ also computes $T$ with respect to $(E,F)$.   So, to prove Theorem \ref{thm:main.1}, it suffices to prove the following.

\begin{theorem}\label{thm:comp.linear.ext}
If $p$ is a computable real so that $p \geq 1$ and $p \neq 2$, and if $F$ is an effective generating set for $\ell^p$, then 
with respect to $F$ the halting problem computes a sequence of disjointly supported unit vectors that generate $\ell^p$.
\end{theorem}

Our main tool for producing such a sequence of unit vectors is the concept of a \emph{disintegration} which we 
define now.  To begin, fix a real $p \geq 1$.  Suppose $S \subseteq \omega^{<\omega}$ and $\phi : S \rightarrow \ell^p$.  We say that $\phi$ is a \emph{reverse-order homomorphism} if $\phi(\tau) \preceq \phi(\rho)$ whenever $\tau, \rho \in S$ are such that $\tau \supset \rho$.   (Recall from Subsection \ref{subsec:background.fa} that $f \preceq g$ if and only if $f(n) = 0$ whenever $f(n) \neq g(n)$.)  We say that $\phi$ is a \emph{strong reverse-order homomorphism} if 
it is a reverse-order homomorphism with the additional feature that it maps incomparable nodes to 
disjointly supported vectors.   
Accordingly, an injective (strong) reverse-order homomorphism will be called a (strong) reverse-order monomorphism.  

Suppose $S$ is a subtree of $\omega^{<\omega}$. When $\nu$ is a nonterminal node of $S$, let $\nu^+_S$ denote the set of all children of $\nu$ in $S$.  Call a map $\phi : S \rightarrow \ell^p$ \emph{summative} if 
$\phi(\nu) = \sum_{\nu' \in \nu^+_S} \phi(\nu')$ whenever $\nu$ is a nonterminal node of $S$.  
A \emph{disintegration} is a summative strong reverse-order monomorphism whose range generates $\ell^p$.

That disintegrations exist is easy to see; e.g. set $\phi(\emptyset) = \sum_n 2^{-n} e_n$ and set 
$\phi((n)) = 2^{-n} e_n$.  The challenge is to produce a disintegration that is computable with respect to an effective generating set $F$ (in the sense that there is an algorithm that given a $\nu \in S$ and a $k \in \N$ computes an $f \in \mathcal{L}_{\Q(i)}(F)$ so that $\norm{\phi(\nu) - f}_p < 2^{-k}$).  Accordingly, in Section \ref{sec:computable}, we prove the following.

\begin{theorem}\label{thm:comp.disintegration}
If $p \geq 1$ is a computable real besides $2$, and if $F$ is an effective generating set for $\ell^p$, then 
there is a disintegration of $\ell^p$ that is computable with respect to $F$.
\end{theorem}

How does possession of a disintegration $\phi : S \rightarrow \ell^p$ that is computable 
with respect to an effective generating set $F$ help us to prove Theorem \ref{thm:comp.linear.ext}?
Intuitively, to define $g_n$ we want to use the halting set to compute the limit of $\phi(\nu)$ as $\nu$
descends some carefully chosen branch of $S$.   To see how we choose these 
branches, we now define the concept of an almost norm-maximizing chain.  When $\nu$ is a non-root node of  $\omega^{< \omega}$, let $\nu^-$ denote its parent.

\begin{definition}\label{def:almost.norm.max}
Suppose $\phi : S \rightarrow \ell^p$ is a disintegration.  
\begin{enumerate}
	\item If $\nu$ is a non-root node of $S$, then we say $\nu$ is an \emph{almost norm-maximizing child} of 
	$\nu^-$ if $\norm{\phi(\mu)}_p^p \leq \norm{\phi(\nu)}_p^p + 2^{-(|\nu|+1)}$ whenever 
	$\mu$ is a child of $\nu^-$ in $S$.
	
	\item A chain $\mathcal{C} \subseteq S$ is \emph{almost norm-maximizing} if 
	every nonterminal node in $\mathcal{C}$ has an almost norm-maximizing child in $\mathcal{C}$.
\end{enumerate}
\end{definition}

In Section \ref{sec:classical} we prove the following.

\begin{theorem}\label{thm:an.chains}
Suppose $\phi : S \rightarrow \ell^p$ is a disintegration.
\begin{enumerate}
	\item If $\mathcal{C} \subseteq S$ is an almost norm-maximizing chain, then 
the $\preceq$-infimum of $\phi[\mathcal{C}]$ exists and is either $\mathbf{0}$ or an atom of $\preceq$.  
Furthermore, $\inf \phi[\mathcal{C}]$ is the limit in the $\ell^p$ norm of $\phi(\nu)$ as $\nu$ traverses 
the nodes in $\mathcal{C}$ in increasing order.
\label{thm:an.chains::itm:infimum}

	\item If $\{\mathcal{C}_n\}_{n = 0}^\infty$ is a partition of 
$S$ into almost norm-maximizing chains, then $\inf \phi[\mathcal{C}_0]$, $\inf \phi[\mathcal{C}_1]$, $\ldots$ 
are disjointly supported vectors that generate $\ell^p$.
\label{thm:an.chains::itm:partition}
\end{enumerate}
\end{theorem}

And, in Section \ref{sec:computable}, we prove:

\begin{theorem}\label{thm:comp.partition}
Suppose $\phi : S \rightarrow \ell^p$ is a disintegration that is computable with respect to an effective
generating set $F$.  Then, there is a computable partition of $S$ into c.e. almost norm-maximizing chains.
\end{theorem}

So, to prove Theorem \ref{thm:comp.linear.ext} we first compute with respect to $F$ a disintegration 
$\phi : S \rightarrow \ell^p$, and then compute a partition of $S$ into c.e. almost norm-maximizing chains 
$\mathcal{C}_0$, $\mathcal{C}_1$, $\ldots$.  Set $g_n = \inf \phi[\mathcal{C}_n]$.  Note that 
$\norm{g_n}_p$ is a right-c.e. real.   Thus, the halting set computes $\norm{g_n}_p$ from $n$.  
Since $g_n \preceq \phi(\nu)$ for all $\nu \in \mathcal{C}_n$, it follows that the halting set computes $g_n$ with respect to $F$ (since $\norm{\phi(\nu) - g_n}_p^p = \norm{\phi(\nu)}_p^p - \norm{g_n}_p^p$ for all $\nu \in \mathcal{C}_n$).  
We can also use the halting set to enumerate all values of $n$ for 
which $g_n \neq \mathbf{0}$; denote these $n_0 < n_1 < \ldots$.  Then, $\{\norm{g_{n_j}}^{-1} g_{n_j}\}_{j = 0}^\infty$ is a disjointly supported sequence of unit vectors that generates $\ell^p$, and the halting set 
computes $\{g_{n_j}\}_{j = 0}^\infty$ with respect to $F$.  

Let us now return to the proof of Theorem \ref{thm:comp.disintegration}.  
The idea is to construct a disintegration $\phi$ via an ascending sequence of partial disintegrations 
that are computable uniformly with respect to $F$.  Specifically, we define a 
 \emph{partial disintegration} to be a strong order monomorphism $\psi : S- \{\emptyset\} \rightarrow \ell^p$ where 
$S$ is a finite non-empty subtree of $\omega^{< \omega}$.  
We say a partial disintegration $\psi_2$ \emph{extends} a partial disintegration $\psi_1$ if 
$\dom(\psi_1) \subseteq \dom(\psi_2)$ and if $\psi_2(\nu) = \psi_1(\nu)$ for all $\nu \in \dom(\psi_1)$.

Let $F = \{f_0, f_1, \ldots\}$ be a generating set for $\ell^p$.  If $F$ is an effective generating set, then it is quite easy to produce a partial disintegration that is computable with respect to $F$.  Namely, 
set $S = \{\emptyset\}$ and $\psi = \emptyset$.  
Of course, this partial disintegration does not do much for us and is quite far from being a disintegration.  
Accordingly, when $\psi : S- \{\emptyset\} \rightarrow \ell^p$ is a partial disintegration, we define 
the \emph{success index of $\psi$} (with respect to $F$) to be the largest integer $N$ so that 
$d(f_j, \langle \ran(\psi) \rangle) < 2^{-N}$ whenever $0 \leq j < N$ and 
\[
\norm{ \psi(\nu) - \sum_{\nu' \in \nu^+_S} \psi(\nu')}_p < 2^{-N}
\]
whenever $\nu$ is a non-root nonterminal node of $S$. 

Here is how we can glue an ascending sequence of partial disintegrations into a disintegration.  
Suppose $\psi_0, \psi_1, \ldots$ is an ascending sequence of partial disintegrations (in the sense that $\psi_{n+1}$ extends $\psi_n$) so that 
the success index of $\psi_{n+1}$ is larger than $n$ for all $n$.
Set $S = \{\emptyset\} \cup \bigcup_n \dom(\psi_n)$, and set 
\[
\psi(\nu) = 2^{-\nu(0)} \lim_n (\norm{\psi_n(\nu(0))}_p + 1)^{-1}\psi_n(\nu)
\]
 for all $\nu \in S - \{\emptyset\}$.  Set $\psi(\emptyset) = \sum_{(a) \in S} \psi((a))$.  Then, $\psi$ is a disintegration.  Such a chain of partial disintegrations can be obtained by a fairly straightforward application of the following which is proven in Section \ref{sec:computable}.

\begin{theorem}\label{thm:comp.partial.disintegration.ext}
Suppose $F$ is an effective generating set for $\ell^p$ where $p \geq 1$ is a computable real besides $2$.  
If $N,k \in \N$, and if $\phi : S - \{\emptyset\} \rightarrow \ell^p$ is a partial disintegration that is computable with respect to $F$, then there exists a map $\mathbf \phi' : S- \{\emptyset\} \rightarrow \ell^p$ so that 
\[
\max\{\phi(\nu) - \phi'(\nu)\ : \nu \in S - \{\emptyset\} \} < 2^{-k}
\]
 and so that $\phi'$ extends to a partial disintegration $\psi$ that is computable with respect to $F$ and whose success index with respect to $F$ 
is larger than $N$.  Furthermore, 
an index of $\psi$ can be computed from $N$, $k$ and an index of $\phi$. 
\end{theorem}

\section{Classical world}\label{sec:classical}

\subsection{Results on disintegrations and partial disintegrations}\label{subec:disint}

The following is a preliminary step to proving Theorem \ref{thm:an.chains}.

\begin{proposition}\label{prop:limit.desc.chain}
If $g_0 \succ g_1 \succ \ldots$ are vectors in $\ell^p$, then 
$\lim_n g_n$ exists pointwise and in the $\ell^p$-norm and is the $\preceq$-infimum of $\{g_0, g_1, \ldots \}$.
\end{proposition}

\begin{proof}
 Let 
\[
S = \bigcap_n \supp(g_n).
\]
Set $g = g_0 \cdot \chi_S$.  Since $g_{n+1} \preceq g_n$, it follows that $g$ is the pointwise limit of 
$\{g_n\}_n$.

We claim that $|g_n(t) - g(t)| \leq |g_0(t)|$ for all $t$.  For, either $g_n(t) = g(t)$, or $g(t) = 0$ and $g_n(t) = g_0(t)$.  
By regarding summation as integration with respect to the counting measure, it now follows from the Dominated Convergence Theorem that 
$\lim_{n \rightarrow \infty} \norm{g_n - g}_p = 0$.  

Suppose $h \preceq g_n$ for all $n$.  Thus, as discussed in Subsection \ref{subsec:background.fa}, $\sigma_1(g_n - h, h) = 0$ for all $n$.  Since $\{g_n\}_{n = 0}^\infty$ converges to $g$ in the $\ell^p$-norm, $\sigma_1(g - h, h) = 0$; that is $h \preceq g$.  Thus, 
$g = \inf \{g_0, g_1, \ldots\}$.
\end{proof}

\begin{proof}[Proof of Theorem \ref{thm:an.chains}:]
(\ref{thm:an.chains::itm:infimum}): Suppose $\mathcal{C} \subseteq S$ is an almost norm-maximizing chain.  It follows from Proposition \ref{prop:limit.desc.chain} that the $\preceq$-infimum of $\phi[\mathcal{C}]$ exists; let $g$ denote this infimum.  By way of contradiction, suppose $j_0, j_1 \in \supp(g)$ and $j_0 \neq j_1$.   
Since $\phi$ maps incomparable nodes to disjointly supported vectors, whenever $\nu \in S$, 
the support of $\phi(\nu)$ contains both $j_0$ and $j_1$ if it contains either one of them.  
Since $\phi$ is reverse-order preserving, if $j_0$ and $j_1$ belong to the support of $\phi(\nu)$, 
then $\phi(\nu)(j_0) = g(j_0)$ and $\phi(\nu)(j_1) = g(j_1)$.  But, since $\phi$ is a disintegration, 
the range of $\phi$ generates $\ell^p$- a contradiction.  
Thus, $g$ is either $\mathbf{0}$ or an atom.\\

(\ref{thm:an.chains::itm:partition}): Suppose $\mathcal{C}_0$, $\mathcal{C}_1$, $\ldots$ is a partition of 
$S$ into almost norm-maximizing chains.  By part (\ref{thm:an.chains::itm:infimum}), 
$\inf \phi[\mathcal{C}_n]$ exists for each $n$; let $h_n = \inf \phi [\mathcal{C}_n]$.  

We first claim that for every $j$, there is a $k$ so that $j$ belongs to the support of $h_k$.  
If there is an atom in $\ran(\phi)$ whose support contains $j$, then there is nothing to prove.  So, suppose $j$ does not belong to the support of any atom in $\ran(\phi)$.  

We claim that there is a $\nu \in S$ so that $j \in \supp(\phi(\nu))$ and $|\nu| = 1$.
For otherwise, $j \not \in \supp(g)$ for all $g \in \ran(\phi)$.  But, since $\phi$ is a disintegration, $\ran(\phi)$ generates $\ell^p$- a contradiction.

Since $\phi$ is a disintegration, if $\nu$ is a nonterminal node of $S$, then $\phi(\nu) = \sum_{\nu' \in \nu^+_S} \phi(\nu')$.   It thus follows by induction that for each $s$, $j$ belongs to the support of a $\phi(\nu)$ so that $|\nu| = s$;  since $\phi$ maps incomparable nodes to disjointly supported vectors, $\nu$ is unique and accordingly we denote it by $\nu_s$.   Let $g_s = \phi(\nu_s)$.  
Again, since $\phi$ maps incomparable nodes to disjointly supported vectors, $\nu_{s+1} \supset \nu_s$ for all $s$.  Since $\phi$ is a reverse-order monomorphism, $g_{s+1} \prec g_s$ for all $s$.  Thus, $g_s(j) = g_0(j) \neq 0$ for all $s$.  

Now, for each $s$, let $k_s$ denote the $k$ so that $g_s \in \phi[\mathcal{C}_k]$.  
We claim that $\lim_s k_s$ exists.  By way of contradiction suppose otherwise.  Let $s_0 < s_1 < \ldots$ be the increasing enumeration of all values of $s$ for which $k_s \neq k_{s+1}$.  
Since $\nu_{s_m + 1} \supset \nu_{s_m}$, $\nu_{s_m}$ is a nonterminal node of $S$.
Therefore, since $\mathcal{C}_{k_{s_m}}$ is almost norm-maximizing, it must contain a child of $\nu_{s_m}$ in $S$; let $\mu_m$ denote this child and let $\lambda_m = \phi(\mu_m)$.  Thus, $\lambda_m \prec g_{s_m}$ and the supports of $\lambda_m$ and $g_{s_m + 1}$ are disjoint (since $\mu_m$ and $\nu_{s_m + 1}$ are distinct nodes at the same level of $S$).  In addition, since $\mu_m$ is an almost norm-maximizing child of $\nu_{s_m}$, $|g_0(j)|^p = |g_{s_m+1}(j)|^p \leq \norm{\lambda_m}^p_p + 2^{-s_m}$.  Since $\lambda_{m + r} \preceq g_{s_{m+r}} \preceq g_{s_m + 1}$, the supports of $\lambda_m$ and $\lambda_{m+r}$ are disjoint if $r > 0$.  That is to say, $\supp(\lambda_m) \cap \supp(\lambda_{m'}) = \emptyset$ whenever $m \neq m'$.  Thus, 
$\infty = \sum_m \norm{\lambda_m}^p_p \leq \norm{g_0}^p_p$- a contradiction.
Thus, $k := \lim_s k_s$ exists, and so $j$ belongs to the support of $h_k$.  Moreover, it follows from 
part (\ref{thm:an.chains::itm:infimum}) that $h_k$ is a nonzero scalar multiple of $e_j$.  
It then follows that $h_0, h_1, \ldots$ generate $\ell^p$.  

Finally, we claim that $h_0, h_1, \ldots$ are disjointly supported.  For, suppose $k \neq k'$.  
It suffices to show that there are incomparable nodes $\nu, \nu'$ so that $\nu \in \mathcal{C}_k$ and $\nu' \in \mathcal{C}_{k'}$.  
If there is an integer $s$ so that $\mathcal{C}_k$ and $\mathcal{C}_{k'}$ both contain a node of length $s$, 
then we may choose $\nu$ and $\nu'$ to be these nodes (since $\mathcal{C}_k \cap \mathcal{C}_{k'} = \emptyset$).  
So, suppose there is no $s$ so that $\mathcal{C}_k$ and $\mathcal{C}_{k'}$ both contain a node of length $s$.  Let $\mu$ be the $\subseteq$-minimal node in $\mathcal{C}_k$ and let $\mu'$ be the $\subseteq$-minimal node in $\mathcal{C}_{k'}$.  Let $s = |\mu|$, and let $s' = |\mu'|$.   Thus, $s \neq s'$.  Without loss of generality, assume $s < s'$.  
This entails that $\mathcal{C}_k$ contains a terminal node of $S$; let $\tau$ denote this node and let $t = |\tau|$.  
Thus, $h_k = \phi(\tau)$.  Furthermore, $t < s'$.  Let $\mu''$ denote the length $t$ ancestor of $\mu'$.  Since $\tau$ is a terminal node of $S$, $\tau$ and $\mu''$ are incomparable.  Thus, $\tau$ and $\mu'$ are incomparable.
\end{proof}

For the sake of proving Theorem \ref{thm:comp.partition}, we prove the following existence result.  

\begin{proposition}\label{prop:max}
If $\phi : S \rightarrow \ell^p$ is a disintegration, and if $\nu$ is a nonterminal node of $S$, then 
\[
\max\{ \norm{\phi(\nu')}^p_p\ :\ \nu' \in \nu^+_S\}
\]
exists.
\end{proposition}

\begin{proof}
Since $\phi$ is a disintegration,  $\phi(\nu) = \sum_{\nu' \in \nu^+_S} \phi(\nu')$.
Since $\phi$ maps incomparable nodes to disjointly supported vectors it follows that 
\[
\sum_{\nu' \in \nu^+_S} \norm{\phi(\nu')}^p_p = \norm{\phi(\nu)}^p_p < \infty.
\]
Therefore, there is a finite set $\{\nu'_0, \ldots, \nu'_t\} \subseteq \nu^+_S$ so that 
\[
\norm{\phi(\nu')}^p_p \leq \max\{\norm{\phi(\nu'_0)}^p_p, \ldots, \norm{\phi(\nu'_t)}^p_p\}
\]
 whenever $\nu' \in \nu^+_S - \{\nu_0', \ldots, \nu_t'\}$.  Thus, 
 \[
 \sup\{\norm{\phi(\nu')}^p_p\ :\ \nu' \in \nu^+_S\} = \max\{\norm{\nu'_0}^p_p, \ldots \norm{\nu'_t}^p_p\}.
 \]
\end{proof}

For the sake of proving Theorem \ref{thm:comp.partial.disintegration.ext}, we prove the following existence theorem; Theorem \ref{thm:comp.partial.disintegration.ext} will then be demonstrated by a search procedure.  Recall that the success index of a partial disintegration, which was defined in Section \ref{sec:overview}, measures how close a partial disintegration is to being a disintegration.

\begin{theorem}\label{thm:partial.disintegration.ext}
Suppose $F = \{f_0, f_1, \ldots\}$ is a generating set for $\ell^p$.  
If $\phi : S- \{\emptyset\} \rightarrow \ell^p$ is a partial disintegration, and if $N_0 \in \N$, then $\phi$ extends to a 
partial disintegration $\psi$ whose success index (with respect to $F$) is larger than $N_0$.  
\end{theorem}

\begin{proof}
There is a nonnegative integer $N_1$ so that $d(f_j, \langle e_0, \ldots, e_{N_1}\rangle) < 2^{-N_0}$ 
whenever $0 \leq j < N_0$ and so that 
$d(\phi(\nu), \langle e_0, \ldots, e_{N_1}\rangle) < 2^{-N_0}$ whenever $\nu \in S - \{\emptyset\}$.  

When $0 \leq k \leq N_1$, let $\nu_k = \emptyset$ if 
$k \not \in \bigcup_{\nu \in S} \supp(\phi(\nu))$; otherwise let $\nu_k$ denote the $\subseteq$-maximal node in $S$ 
so that $k \in \supp(\phi(\nu))$.  

Intuitively, we define $\psi$ so that its range includes nonzero scalar multiples of each of $e_0, \ldots, e_{N_1}$.  We first define the domain of $\psi$.  Let 
\[
S' = S \cup \{\nu_k\cat(k + \# S)\ :\ 0 \leq k \leq N_1\ \wedge\ (\nu_k = \emptyset\ \vee\ \#\supp(\phi(\nu_k)) \geq 2)\}.
\]
(Here, $\# A$ denotes the cardinality of $A$, and ${\ }\cat$ denotes concatenation.)  Let $\psi(\nu) = \phi(\nu)$ if $\nu \in S - \{\emptyset\}$.  
Suppose $\nu_k\cat(k + \#S) \in S'$.  Let 
\[
\psi(\nu\cat(k + \#S)) = \left\{ \begin{array}{cc}
								\phi(\nu_k) \cdot e_k & \mbox{if $\nu_k \neq \emptyset$}\\
								e_k & \mbox{otherwise}\\
								\end{array}
								\right.
\]
Thus, by construction, $\psi$ is a partial disintegration, and $\psi$ extends $\phi$.  

We claim that the success index of $\psi$ is at least as large as $N_0$.  
For, by construction, $e_0, \ldots, e_{N_1} \in \langle \ran(\psi) \rangle$.  
Thus, by choice of $N_1$, $d(f_j, \langle \ran(\psi) \rangle) < 2^{-N_0}$ whenever $0 \leq j < N_0$.  
Suppose $\nu$ is a nonterminal node of $S'$.  Thus, by definition of $S'$, $\nu \in S$.  
We show that 
\begin{equation}
\norm{\psi(\nu) - \sum_{\nu' \in \nu^+_{S'} } \psi(\nu')}_p < 2^{-N_0}.\label{ineqn:sum}
\end{equation}
Since $\nu \in S$, $\psi(\nu) = \phi(\nu)$.
By choice of $N_1$, $\norm{\phi(\nu) - \phi(\nu)\cdot \chi_{\{0, \ldots, N_1\}}}_p < 2^{-N_0}$.
By definition of $S'$, whenever $0 \leq k < N_1$ and $k \in \supp(\phi(\nu))$, $k$ belongs to the support 
of $\phi(\nu')$ for some child $\nu'$ of $\nu$ in $S'$; furthermore $\phi(\nu) \cdot e_k \preceq \psi(\nu')$.
The inequality (\ref{ineqn:sum}) follows.
\end{proof}

\subsection{On the distance to the nearest strong reverse-order homomorphism}\label{subsec:distance.strong}

Let $S$ be a finite nonempty subset of $\omega^{<\omega}$.  Define $\ell^p_S$ to be the space of all functions that map $S$ into $\ell^p$.  
When $\phi \in \ell^p_S$, define $\norm{\phi}_S$ to be $\max \{\norm{\phi(\nu)}_p\ :\ \nu \in S\}$.  
Thus, $\norm{\ }_S$ is a norm on $\ell^p_S$ under which $\ell^p_S$ is a Banach space.  

Suppose $\phi \in \ell^p_{S - \{\emptyset\}}$ is a partial disintegration, and let $S' \supseteq S$ be a finite subtree of $\omega^{< \omega}$.  
Define $\mathcal{H}_{\phi, S'}$ to be the set of all strong reverse-order homomorphisms 
$\psi \in \ell^p_{S' - \{\emptyset\}}$ so that $\psi(\nu)  =\phi(\nu)$ whenever $\nu$ is a non-root node of $S$.  Thus, $\mathcal{H}_{\phi, S'}$ is a closed subset of 
$\ell^p_{S'- \{\emptyset\}}$.  For the sake of a search procedure we will employ in the proof of Theorem \ref{thm:comp.partial.disintegration.ext}, we wish to find a reasonable upper bound on $d(\psi, \mathcal{H}_{\phi, S'})$ in terms of $\phi, \psi$.   

When $p \neq 2$, set $c_p = |4 - 2\sqrt{2}^p|^{-1}$. When $z,w \in \C$ set $\sigma(z,w)  =  c_p \sigma_1(z,w)$.
As a first step toward bounding $d(\psi, \mathcal{H}_{\phi, S'})$ above, we prove the following sharpening 
of an inequality due to J. Lamperti \cite{Lamperti.1958}.

\begin{theorem}\label{thm:lamperti}
Suppose $p \geq 1$ and $p \neq 2$.  Then, 
\begin{equation}
\min\{|z|^p, |w|^p\} \leq \sigma(z,w) \label{ineq:lamperti}
\end{equation}
for all $z,w \in \C$.
Furthermore, if $1 \leq p < 2$, then 
\[
2|z|^p + 2|w|^p - |z+w|^p - |z-w|^p \geq 0
\]
and if $2 < p$ then 
\[
2|z|^p + 2|w|^p - |z+w|^p - |z-w|^p \leq 0.
\]
\end{theorem}

\begin{proof}
Without loss of generality, assume $0 < |z| \leq |w|$.  Set $w/z = t e^{i \theta}$ where $t \geq 1$.   Then, (\ref{ineq:lamperti}) reduces to 
\[
1 \leq \frac{|2 + 2t^p - |1 + t e^{i \theta}|^p - |1 - te^{i\theta}|^p|}{|4 - 2\sqrt{2}^p|}.
\]
This leads to consideration of the function
\[
f_p(\theta,t) := \left\{\begin{array}{cc}
2 + 2t^p - |1 + t e^{i \theta}|^p - |1 - t e^{i \theta}|^p & 1 \leq p < 2\\
|1 + t e^{i\theta}|^p + |1 - t e^{i \theta}|^p - 2t^p - 2 & p > 2\\
\end{array}
\right.
\]
We show that 
\[
\min_{t \geq 1} f_p(\theta, t) = |4 - 2 \sqrt{2}^p|.
\]
We note that $f_p(\theta + \pi,t) = f_p(\theta, t)$.  So, we restrict attention to values of $\theta$ between $0$ and $\pi$.  We use basic multivariable calculus to minimize $f_p(\theta,t)$ in the region $[0,\pi] \times [1, \infty)$.  To this end, we first note that 
\[
\frac{\partial}{\partial t} |1 \pm t e^{i \theta}| = \frac{t \pm \cos(\theta)}{|1 \pm t e^{i\theta}|} 
\]
and that
\[
\frac{\partial}{\partial \theta}  |1 \pm t e^{i \theta}| = \frac{\mp t \sin(\theta)}{|1 \pm t e^{i \theta}|}
\]
It follows that when $1 \leq p < 2$: 
\begin{eqnarray*}
\frac{\partial f_p}{\partial t}(\theta,t) & = & 2pt^{p-1} - p[(t - \cos(\theta))|1 - te^{i \theta}|^{p-2} + (t + \cos(\theta))|1 + t e^{i \theta}|^{p-2}]\\
\frac{\partial f_p}{\partial \theta}(\theta,t) & = & pt\sin(\theta)[ |1 + te^{i \theta}|^{p-2} - |1 - te^{i \theta}|^{p-2}].
\end{eqnarray*}
The signs are reversed when $p > 2$. 

So, when $0 < \theta_0 < \pi$ and $t_0 \geq 1$, 
\begin{eqnarray*}
\frac{\partial f_p}{\partial \theta}(\theta_0, t_0) = 0 & \Leftrightarrow & |1 + t_0 e^{i \theta_0}| = |1 - te^{i \theta_0}|\\
& \Leftrightarrow & \theta_0 = \pi/2.
\end{eqnarray*}
We now claim that $\frac{\partial f_p}{\partial t}(\pi/2, t_0) > 0$ when $t_0 \geq 1$.  We first consider the case $1 \leq p < 2$.  We have
\[
\frac{\partial f_p}{\partial t}(\pi/2, t) = 2pt^{p-1} - 2pt|1 + ti|^{p-2}.
\]
Since $t < |1 + ti|$ and $p - 2 < 0$, $t^{p-2} > |1 + ti|^{p-2}$.  Thus, $\frac{\partial f_p}{\partial t}(\pi/2, t_0) > 0$.  The case $2 < p$ is symmetric; the signs are merely reversed and $p - 2 > 0$.

We next claim that $\frac{\partial f_p}{\partial t}(0,t) \geq 0$ if $t \geq 1$.  We first consider the case $1 \leq p < 2$.  In this case the claim reduces to 
\[
2 \geq \left( \frac{t-1}{t} \right)^{p-1} + \left( \frac{t+1}{t} \right)^{p-1}.
\]
Since $0 \leq p-1 < 1$, $x \mapsto x^{p-1}$ is concave.  Thus, 
\[
1 = \left( \frac{\frac{t-1}{t} + \frac{t+1}{t}}{2} \right)^{p-1} \geq \frac{1}{2} \left[\left(\frac{t-1}{t} \right)^{p-1} + \left( \frac{t + 1}{t} \right)^{p-1}\right].
\]
This verifies the claim when $1 \leq p < 2$.  The case $2 < p$ is symmetric: signs are reversed and the function $x \mapsto x^{p-1}$ is convex.

Thus, $\frac{\partial f_p}{\partial t} (\pi, t) \geq 0$ if $1 \leq t$.  

So, let $t_0 > 1$, and let $R$ denote the rectangle $[0, \pi] \times [1, t_0]$.  It follows from what has just been shown that the
minimum of $f_p$ on $R$ is achieved on the lower line segment $[0, \pi] \times \{1\}$.  Moreover, it is achieved at one of the points $(0,1)$, $(\pi/2, 1)$, $(\pi,1)$.  $f_p(0,1) = f_p(0, \pi) = |2 - 2^p|$ and $f_p(0, \pi/2) = |4 - 2 \sqrt{2}^p|$.  Since $p \neq 2$, it follows that $|4 - 2^p| > |4 - 2\sqrt{2}^p|$.  Thus, the minimum of $f_p$ on $R$ is $|4 - 2 \sqrt{2}^p|$.  Since $t_0$ can be any number larger than $1$, the minimum of $f_p$ on $[0,\pi] \times [1, \infty)$ is $|4 - 2\sqrt{2}^p|$.

The theorem now follows.
\end{proof}

When $\psi \in \ell_S^p$, set 
\[
\sigma(\psi)  =  \sum_{\nu | \nu'} \sigma(\psi(\nu), \psi(\nu')) + \sum_{\nu' \supset \nu} \sigma(\psi(\nu') - \psi(\nu), \psi(\nu'))
\]
where $\nu, \nu'$ range over all nodes of $S$ and $\nu | \nu'$ denotes that $\nu$ and $\nu'$ are incomparable.  
Note that $\sigma : \ell^p_S \rightarrow [0, \infty)$ is continuous and $\sigma(\psi) = 0$ if and only if 
$\psi$ is a strong order homomorphism.  The following is the main result of this subsection.

\begin{theorem}\label{thm:distance.to.ext}
Suppose $\phi \in \ell^p_{S - \{\emptyset\}}$ is a partial disintegration.  Suppose $\psi \in \ell^p_{S' - \{\emptyset\}}$ where 
$S' \supseteq S$ is a finite subtree of $\omega^{<\omega}$ so that each node of $S'$ extends a node of $S$.  Then, 
\[
d(\psi, \mathcal{H}_{\phi, S'})^p \leq  \norm{\phi|_{S - \{\emptyset\}} - \psi|_{S - \{\emptyset\}} }^p_{S - \{ \emptyset\}} + 2^p\sigma(\phi \cup \psi|_{S' - S}).
\]
\end{theorem}

\begin{proof}  
Let $\psi_0 = \phi \cup \psi|_{S' - S}$.
Set
\begin{eqnarray*}
\hat{\sigma}(\psi_0)(n) & = & \sum_{\nu | \nu'} \min\{|\psi_0(\nu)(n)|^p, |\psi_0(\nu')(n)|^p\}\\
& & + \sum_{\nu' \supset \nu} \min\{|\psi_0(\nu')(n) - \psi_0(\nu)(n)|^p, |\psi_0(\nu')(n)|^p\}
\end{eqnarray*}
where $\nu, \nu'$ range over the nodes of $S' - \{\emptyset\}$.  
Thus, by Theorem \ref{thm:lamperti}, $\sum_n \hat{\sigma}(\psi_0)(n) \leq \sigma(\psi_0)$.

We now construct $\psi_1$.
If $\nu \in S - \{\emptyset\}$, set $\psi_1(\nu) = \phi(\nu)$.  If $\nu \in S' - S$, and if $n \in \N$, set 
\[
\psi_1(\nu)(n) = \left\{ \begin{array}{cc}
							\psi_1(\nu^-)(n) & \mbox{if $|\psi_0(\nu)(n)|^p > \hat{\sigma}(\psi_0)(n)$ and $\nu^- \neq \emptyset$}\\
							\psi(\nu)(n) & \mbox{if $|\psi_0(\nu)(n)|^p > \hat{\sigma}(\psi_0)(n)$ and $\nu^- = \emptyset$}\\
							0 & \mbox{otherwise.} \\
							\end{array}
							\right.
\]

By construction $\psi_1$ is a reverse-order homomorphism.  We show it is a strong reverse-order homomorphism.  
Suppose $\nu, \nu' \in S'$ are incomparable.  Since $\psi_1$ is a reverse-order homomorphism, it suffices to consider the case where $\nu, \nu' \not \in S$.  Suppose $\psi_1(\nu)(n) \neq 0$.  
Then, $|\psi_0(\nu)(n)|^p > \hat{\sigma}(\psi_0)(n)$.  So, 
$|\psi_0(\nu)(n)|^p > |\psi_0(\nu')(n)|^p$.  Thus, $\hat{\sigma}(\psi_0)(n) \geq |\psi_0(\nu')(n)|^p$.  
Hence, $\psi_1(\nu')(n) = 0$.  Thus, $\psi_1(\nu)$ and $\psi_1(\nu')$ are disjointly supported. 

If $\nu \in S - \{\emptyset\}$, then $\norm{\psi(\nu) - \psi_1(\nu)}_p = \norm{\psi(\nu) - \phi(\nu)}_p$.  
Suppose $\nu \in S' - S$.  Let $n \in \N$.  It suffices to show that 
$|\psi(\nu)(n) - \psi_1(\nu)(n)|^p \leq 2^p \hat{\sigma}(\psi_0)(n)$.  
We first consider the case $\psi_1(\nu)(n) = 0$.  
We can assume $|\psi(\nu)(n)|^p > \hat {\sigma}(\psi_0)(n)$.  
Thus, there exists $\mu \subset \nu$ so that $\psi_1(\mu)(n) = 0$; take the least such $\mu$.  
Therefore $|\psi_0(\mu)(n)|^p \leq \hat{\sigma}(\psi_0)(n)$.  
On the other hand, $|\psi_0(\mu)(n) - \psi_0(\nu)(n)|^p \leq \hat{\sigma}(\psi_0)(n)$.  
Therefore, $|\psi(\nu)(n)|^p \leq 2^p \hat{\sigma}(\psi_0)(n)$.  
Now, suppose $\psi_1(\nu)(n) \neq 0$.  Then, $|\psi(\nu)(n)|^p > \hat{\sigma}(\psi_0)(n)$.  
Let $\nu_0$ denote the maximum prefix of $\nu$ that belongs to $S$ or has length $1$.  Then, 
$\psi_1(\nu)(n) = \psi_0(\nu_0)(n)$.  
However, $|\psi(\nu)(n) - \psi(\nu_0)(n)|^p \leq \hat{\sigma}(\psi_0)(n)$.  
\end{proof}

We note that the hypothesis that each node of $S'$ extends a node in $S$ is not superfluous.  
For, let $p = 1$.  
Choose $x > 0$ so that $2\sigma(x, 1) < 1$.
Let $S = \{(0)\}$, and let $\phi = \{(0), x e_0)\}$.  Let $S' = \{(0), (1)\}$, and let 
$\psi = \phi \cup \{ ((1), e_0)\}$.  If $\psi' : S' \rightarrow \ell^1$ is a strong reverse-order homomorphism
that extends $\phi$, then $\norm{\psi' - \psi}_{S'} \geq 1 > 2\sigma(\psi)$.

\section{Computable world}\label{sec:computable}

\subsection{Proof of Theorem \ref{thm:comp.partition}}

Suppose $\phi : S \rightarrow \ell^p$ is a disintegration that is computable with respect to $F$.   
Since $\phi$ is computable, $S$ is c.e.; fix a computable enumeration of $S$, $\{S_t\}_{t \in \N}$.  
We can choose this enumeration so that each $S_t$ is a finite subtree of $\omega^{<\omega}$.

It suffices to show that from a nonterminal node $\mu \in S$ we can compute an almost-norm maximizing 
child of $\mu$ in $S$.  We base the proof of this claim on a sequence of lemmas as follows.  
When $\nu \in S_t$, let 
 $\nu^+_t$ abbreviate $\nu^+_{S_t}$.

\begin{lemma}\label{lm:max.norm.below}
If $\nu$ is a non-root and nonterminal node of $S$, then there are infinitely many numbers $t$ so that 
\begin{equation}
\norm{\phi(\nu)}_p^p - \sum_{\mu \in \nu^+_t} \norm{\phi(\mu)}_p^p < \max\{ \norm{\phi(\mu)}_p^p\ :\ \mu \in \nu^+_t\} \label{ineq:max.norm.below}
\end{equation}
When $t$ is such a number, 
\[
\max\{\norm{\phi(\mu)}_p^p\ :\ \mu \in \nu_t^+\} = \max\{ \norm{\phi(\mu)}_p^p\ :\ \mu \in \nu^+_S\}.
\]
\end{lemma}

\begin{proof}
By Proposition \ref{prop:max}, there is a $\mu_0 \in \nu^+_S$ so that 
\[
\norm{\phi(\mu_0)}_p^p = \max\{\norm{\phi(\mu)}_p^p\ :\ \mu \in \nu^+_S\}.
\]
Since $\phi$ is a disintegration, $\norm{\phi(\mu_0)}_p^p \neq 0$ and 
\[
\lim_{t \rightarrow \infty} \norm{\phi(\nu)}^p_p - \sum_{\mu \in \nu^+_t} \norm{\phi(\mu)}^p_p = 0.
\]
Thus, there are infinitely many numbers $t$ so that (\ref{ineq:max.norm.below}) holds.

Now, suppose $t$ is a number so that (\ref{ineq:max.norm.below}) holds.  By way of contradiction, suppose $\norm{\phi(\mu_0)}^p_p > \max\{\mu \in \nu_t^+\ :\ \norm{\phi(\mu)}^p_p\}$.  Therefore,
\[
\norm{\phi(\nu)}^p_p < \sum_{\mu \in \nu_t^+} \norm{\phi(\mu)}^p_p + \norm{\phi(\mu_0)}^p_p.
\]
Since $\mu_0 \in \nu^+_S$ but $\mu_0 \not \in \nu^+_t$, it follows that $\mu_0$ is incomparable with every node in $\nu^+_t$.  So, since $\phi$ maps incomparable nodes to disjointly supported vectors,  
\[
\sum_{\mu \in \nu_t^+} \norm{\phi(\mu)}^p_p + \norm{\phi(\mu_0)}^p_p \leq \sum_{\mu \in \nu^+_S} \norm{\phi(\nu)}^p_p = \norm{\phi(\nu)}^p_p.
\]
This is a contradiction.  Therefore, $\norm{\phi(\mu_0)}^p_p = \max\{\norm{\phi(\mu)}^p_p\ :\ \mu \in \nu^+_S\}$.
\end{proof}

Since $\phi$ is computable with respect to $F$, $\nu \mapsto \norm{\phi(\nu)}_p$ is computable.  
So, there is a computable function $q : S \times \N \rightarrow \Q$ so that 
$|q(\nu,t) - \norm{\phi(\nu)}^p_p | < 2^{-(t+1)}$.  Set:
\begin{eqnarray*}
m(\nu,t) & = & \min\{q(\nu,t) - 2^{-(t+1)}, 0\}\\
M(\nu,t) & = & q(\nu,t) + 2^{-(t+1)}\\
\Sigma^-(X,t) & = & \sum_{\nu \in X} m(\nu,t)\\
\overline{m}(X, t) & = & \max\{m(\mu,t)\ :\ \mu \in X\}
\end{eqnarray*}
Thus, $m(\nu,t)$ is a lower bound on $\norm{\phi(\nu)}^p_p$, and $M(\nu, t)$ is an upper bound on $\norm{\phi(\nu)}^p_p$.  $\Sigma^-(\nu_t^+, t)$ is a lower bound on $\sum_{\mu \in \nu^+_S} \norm{\phi(\mu)}^p_p$, and $\overline{m}(\nu^+_t, t)$ is a lower bound on $\max\{\norm{\phi(\mu)}^p_p\ :\ \mu \in \nu^+_S\}$.  Also 
$\overline{m}(\nu_t^+, t) + 2^{-t}$ is an upper bound on $\max\{\norm{\phi(\mu)}^p_p\ :\ \mu \in \nu^+_S\}$.

\begin{lemma}\label{lm:max.norm.below.approx}
Suppose $\nu$ is a non-root and nonterminal node of $S$.  Then, there are infinitely many stages $t$ so that $M(\nu,t) -\Sigma^-(\nu^+_t,t) < \overline{m}(\nu_t^+, t)$.  At such a stage $t$, 
\[
0 \leq \max\{\norm{\phi(\mu)}^p_p\ :\ \mu \in \nu^+_S\} - \overline{m}(\nu^+_t, t) < 2^{-t}
\]
\end{lemma}

\begin{proof}
Let $N \in \N$.
By Lemma \ref{lm:max.norm.below}, there is a stage $t_0 > N$ so that 
\[
\norm{\phi(\nu)}^p_p - \sum_{\mu \in \nu_{t_0}^+} \norm{\phi(\mu)}^p_p < \max\{\norm{\phi(\mu)}^p_p\ :\ \mu \in \nu^+_{t_0}\}.
\]
Set $U = \nu_{t_0}^+$.  Then, 
\[
\lim_{t \rightarrow \infty} M(\nu,t) - \Sigma^-(U,t) = \norm{\phi(\mu)}^p_p - \sum_{\mu \in U} \norm{\phi(\mu)}^p_p
\]
and, 
\[
\lim_{t \rightarrow \infty} \overline{m}(\nu^+_t, t) = \max\{\mu \in \nu^+_S\ :\ \norm{\phi(\mu)}^p_p\}.
\]
So, there is a number $t_1 > t_0$ so that 
\[
M(\nu,t_1) - \Sigma^-(U, t_1) < \overline{m}(\nu_{t_1}^+, t_1).
\]
 By definition, $m(\nu, t) \geq 0$.  Since $t_1 > t_0$, $U \subseteq \nu^+_{t_1}$.  Thus, $M(\nu,t_1) -\Sigma^-(\nu^+_{t_1},t_1) < \overline{m}(\nu_{t_1}^+, t_1)$.

Now, suppose $M(\nu,t) -\Sigma^-(\nu^+_t,t) < \overline{m}(\nu_t^+, t)$.
By definition of $M$, $\Sigma^-$, $m$, 
\[
\norm{\phi(\nu)}^p_p - \sum_{\mu \in \nu^+_t} \norm{\phi(\mu)}^p_p \leq M(\nu,t) -\Sigma^-(\nu^+_t,t), 
\]
and 
\[
\overline{m}(\nu_t^+, t) \leq \max_{\mu \in \nu^+_t} \norm{\phi(\mu)}^p_p
\]
So, by Lemma \ref{lm:max.norm.below}, 
\[
\max\{\norm{\phi(\mu)}^p_p\ :\ \mu \in \nu^+_S\} = \max\{ \norm{\phi(\mu)}^p_p\ :\ \mu \in \nu^+_t\}.
\]
Furthermore, 
\[
\max\{\norm{\phi(\nu)}^p_p\ :\ \mu \in \nu^+_t\} < \overline{m}(\nu,t) + 2^{-t}.
\]
This proves the lemma.
\end{proof}

Now, suppose $\mu$ is a nonterminal node of $S$. 
Search for $t>  |\mu| $ so that 
\[
M(\mu,t) -\Sigma^-(\mu^+_t,t) < \overline{m}(\mu_t^+, t).
\]
Then, find $\tau \in \mu_t^+$ so that $m(\tau, t) = \overline{m}(\mu_t^+, t)$.  
Therefore, 
\begin{eqnarray*}
\max\{\norm{\phi(\mu')}^p_p\ :\ \mu' \in \mu^+\} & = & \max\{\norm{\phi(\mu')}^p_p\ :\ \mu' \in \mu_t^+\}\\
& < & m(\tau,t) + 2^{-t}\\
& < & m(\tau,t) + 2^{-|\mu|}\\
& \leq & \norm{\phi(\tau)}^p_p + 2^{-|\mu|}
\end{eqnarray*}
Thus, $\tau$ is an almost norm-maximizing child of $\mu$ in $S$. 

\subsection{Proof of Theorem \ref{thm:comp.partial.disintegration.ext}}

Suppose $F = \{f_0, f_1, \ldots\}$ is an effective generating set for $\ell^p$.  
Let $F^S$ denote the set of all maps from $S$ into $F$.  It follows that 
$F^S$ is an effective generating set for $\ell^p_S$; i.e. $(\ell^p_S, F^S)$ is a computable 
Banach space.  
Furthermore, $\mathcal{L}_{\Q(i)}(F^S)$ coincides with the set of all maps from 
$S$ into $\mathcal{L}_{\Q(i)}(F)$.

We now introduce some notation.  Suppose $S'$ is a finite subtree of $\omega^{< \omega}$ that includes $S$.   
Let:
\begin{eqnarray*}
\mathcal{M}_{S'} & = & \{\psi \in \ell^p_{S'- \{\emptyset\}}\ :\ \mbox{$\psi$ is injective}\}\\
\Delta_{S', N} & = & \{\psi \in \ell^p_{S' - \{\emptyset\}}\ :\ \forall 0 \leq j < N\ d(f_j, \langle \ran(\psi) \rangle) < 2^{-N}\}
\end{eqnarray*}
Let $\mathcal{S}_{S', N}$ denote the set of all $\psi \in \ell^p_{S' - \{\emptyset\}}$ so that 
\[
\norm{\psi(\nu) - \sum_{\nu' \in \nu^+_{S'}} \psi(\nu')}_p < 2^{-N}
\]
whenever $\nu$ is a non-root and nonterminal node of $S'$.  

\begin{lemma}\label{lm:ce.closed.open}
\begin{enumerate}
	\item If each node of $S'$ extends a node of $S$, then $\mathcal{H}_{\phi, S'}$ is c.e. closed uniformly in $\phi, S'$.\label{lm:ce.closed.open::itm:closed}
	
	\item The sets $\mathcal{M}_S$, $\Delta_{S',N}$, $\mathcal{S}_{S',N}$ are c.e. open uniformly in $S, S', N$.  \label{lm:ce.closed.open::itm:open}
\end{enumerate}
\end{lemma}

\begin{proof}
(\ref{lm:ce.closed.open::itm:closed}): When $\psi \in \ell^p_{S' - \{\emptyset\}}$, set 
\[
\mathcal{E}(\psi) =  \norm{\psi'|_{S - \{\emptyset\}} - \phi|_{S - \{\emptyset\}} }_p^p + 2^p\sigma(\phi \cup \psi'|_{S' - S}).
\] 
Therefore, $\psi \in \mathcal{H}_{\phi, S'}$ if and only if $\mathcal{E}(\psi) = 0$.
By Theorem \ref{thm:distance.to.ext}, $d(\psi, \mathcal{H}_{\phi, S'}) \leq \mathcal{E}(\psi)^{1/p}$.  
Since $\phi, p$ are computable, $\mathcal{E}$ is computable.  Thus, by Proposition \ref{prop:bounding.principle}, 
$\mathcal{H}_{\phi, S'}$ is c.e. closed.\\

(\ref{lm:ce.closed.open::itm:open}): When $\psi \in \ell^p_S$, let 
\[
G_1(\psi) = \min\{\norm{\psi(\nu) - \psi(\nu')}_p\ :\ \nu, \nu' \in S\ \wedge\ \nu \neq \nu'\}.
\]
Therefore, $G_1$ is computable with respect to $F^{S'}$.  Since, 
$\mathcal{M}_{S'} = G_1^{-1}[(0, \infty)]$, $\mathcal{M}_{S'}$ is c.e. open.

When $\psi \in \ell^p_{S'}$, let $G_2(\psi)$ denote the minimum of 
\[
\norm{\psi(\nu) - \sum_{\nu' \in \nu^+_{S'}} \psi(\nu')}_p 
\]
as $\nu$ ranges over the nonterminal non-root nodes of $S'$.  It follows that 
$G_2$ is computable with respect to $F^{S'}$ and so 
$\mathcal{S}_{S', N} = G_2^{-1}[(-\infty, 2^{-N})]$ is c.e. open.  

Note that $\psi \in \Delta_{S', N}$ if and only if there exists 
$\beta : \{0, \ldots, N-1\}\ \times\ S' - \{\emptyset\} \rightarrow \Q(i)$ so that 
\[
\norm{f_j - \sum_{\nu \in S' - \{\emptyset\}} \beta(j,\nu) \psi(\nu)}_p < 2^{-N}.
\]
whenever $0 \leq j < N$.  When $\beta : \{0, \ldots, N-1\} \times S' - \{\emptyset\} \rightarrow \Q(i)$, 
set 
\[
\Delta_{S', N, \beta} = \left\{\psi \in \ell^p_{S' - \{\emptyset\}}\ :\ \forall 0 \leq j < N\ \norm{f_j - \sum_{\nu \in S' - \{\emptyset\}} \beta(j, \nu) \psi(\nu)}_p < 2^{-N} \right\}.
\]
Thus, $\Delta_{S', N} = \bigcup_\beta \Delta_{S', N, \beta}$.  
Set 
\[
G_{3, \beta}(\psi) = \max\left\{\norm{f_j - \sum_{\nu \in S' - \{\emptyset\}} \beta(j, \nu) \psi(\nu)}_p\ :\ 0 \leq j < N\right\}.
\]
Therefore, $\Delta_{S', N, \beta} = G_{3, \beta}^{-1}[(-\infty, 2^{-N})]$.  Hence, 
$\Delta_{S', N, \beta}$ is c.e. open uniformly in $S', N, \beta$.  Thus, 
$\Delta_{S', N}$ is c.e. open uniformly in $S', N$.  
\end{proof}

We can now prove Theorems \ref{thm:comp.partial.disintegration.ext} and \ref{thm:comp.disintegration}.  

\begin{proof}[Proof of Theorem \ref{thm:comp.partial.disintegration.ext}:]  For the moment, fix a finite tree $S' \supseteq S$.  When $S' \supseteq S$, let $\pi_{S'}$ denote the canonical projection of $\ell^p_{S' - \{\emptyset\}}$ onto 
$\ell^p_{S - \{\emptyset\}}$, and let 
\[
C_{S'} = \mathcal{H}_{\emptyset, S'} \cap \mathcal{M}_{S'} \cap \Delta_{S,N} \cap \mathcal{S}_{S',N} \cap \pi^{-1}_{S'}[B(\phi; 2^{-k})].
\]
By Theorem \ref{thm:partial.disintegration.ext}, there \emph{is} an $S'$ so that 
$C_{S'} \neq \emptyset$.  Such an $S'$ can be found by an effective search procedure.  Since $\mathcal{H}_{\emptyset, S'}$ is c.e. closed, it follows that 
$C_{S'}$ contains a vector $\psi$ that is computable with respect to $F^{S' - \{\emptyset\}}$ and an index of $\psi$ can be computed from $k$, $N$, and an index of $\phi$.  
\end{proof}

\begin{proof}[Proof of Theorem \ref{thm:comp.disintegration}:]
Let $F = \{f_0, f_1, \ldots\}$.  

Set $S_0 = \{(0)\}$.  Compute $j_0$ so that $f_{j_0} \neq \mathbf{0}$.  Set $\hat{\phi}_0((0)) = f_{j_0}$.  By Lemma \ref{lm:ce.closed.open}, we can compute a $k_0 \in \N$ so that each map in $B(\hat{\phi}_0; 2^{-k_0})$ is injective and never zero.

It now follows from Theorem \ref{thm:comp.partial.disintegration.ext} and Lemma \ref{lm:ce.closed.open} that there is a sequence $\{\hat{\phi}_n\}_n$ of computable partial disintegrations of $\ell^p$ and a sequence $\{k_n\}_n$ of nonnegative integers that have the  following properties. 
\begin{enumerate}
	\item An index of $\hat{\phi}_n$ and a canonical index of $\dom(\hat{\phi}_n)$ can be computed from $n$. 
	
	\item If $S_n =\dom(\phi_n)$, then $S_{n+1} \supset S_n$ and $\norm{ \hat{\phi}_{n+1}|_{S_n} - \hat{\phi}_n}_{S_n} < 2^{-(k_n+1)}$. 
	
	\item Each map in $\overline{B(\hat{\phi}_n; 2^{-k_n})}$ is injective, never zero, and has a success index that is at least $n$.
\end{enumerate}

Let $\phi_{n,t} = \hat{\phi}_{t + n} | S_n$ for all $n,t$.  It follows that $\{\phi_{n,t}\}_t$ is computable
with respect to $F^{S_n - \{\emptyset\}}$; furthermore, an index of this sequence can be computed from $n$.  
It also follows that $\norm{\phi_{n,t+1} - \phi_{n,t}}_{S_n} < 2^{-(k_{n+t} + 1)}$.  Thus, $\phi_n := \lim_t \phi_{n,t}$ is computable with respect to $F^{S_n - \{\emptyset\}}$; furthermore, an index of $\phi_n$ can be computed from $n$.  Also, $\norm{\hat{\phi}_n - \phi_n}_{S_n} \leq 2^{-k_n}$.  Thus, $\phi_n$ is a partial disintegration whose success index is at least $n$.  By definition, $\phi_{n,t+1} \subseteq \phi_{n+1,t}$.  Thus, $\phi_n \subseteq \phi_{n+1}$.  Let $\phi = \bigcup_n \phi_n$.

Let $S = \dom(\phi)$.  For each $\nu \in S$, let 
\[
\psi(\nu) = 2^{-\nu(0)} \norm{\phi((\nu(0))}_p^{-1} \phi(\nu).
\]
Then, let 
\[
\psi(\emptyset) = \sum_{\nu \in \N^1 \cap S} \psi(\nu).
\]
Since $S$ is computable, it follows that $\psi(\emptyset)$ is a computable with respect to $F$.  It then follows that $\psi$ is a disintegration and is computable with respect to $F$.
\end{proof}

\subsection{Proof of Theorem \ref{thm:main.2}}\label{subsec:proof.thm.main.2}

Suppose $p$ is a computable real so that $p \geq 1$ and $p \neq 2$, and assume $C$ is a c.e. set.  
Again, we can reduce to the consideration of surjective linear endomorphisms of $\ell^p$.  
Specifically, it suffices to show there is an effective generating set $F$ for $\ell^p$ so that,  
with respect to $(E, F)$, $C$ computes a surjective linear endomorphism of $\ell^p$ and so that 
any oracle that with respect to $(E,F)$ computes a surjective linear endomorphism of $\ell^p$ also computes $C$.
We demonstrate this as follows.

We can assume $C$ is incomputable.  Without loss of generality, we also assume $0 \not \in C$.  Let $\{c_n\}_{n \in \N}$ be a one-to-one effective enumeration of $C$.  Set
\[
\gamma = \sum_{k \in C} 2^{-k}.
\]
Thus, $0 < \gamma < 1$, and $\gamma$ is an incomputable real.  Set:
\begin{eqnarray*}
f_0 & = & (1 - \gamma)^{1/p} e_0 + \sum_{n = 0}^\infty 2^{- c_n / p} e_{n + 1}\\
f_{n + 1} & = & e_{n + 1}\\
F & = & \{f_0, f_1, f_2, \ldots \}
\end{eqnarray*}
Since $1 - \gamma > 0$, we can use the standard branch of $\sqrt[p]{\ }$.  
We divide the rest of the proof into a sequence of lemmas.

\begin{lemma}\label{lm:EGS}
$F$ is an effective generating set.
\end{lemma}

\begin{proof}
Since 
\[
(1 - \gamma)^{1/p} e_0 = f_0 - \sum_{n =1}^\infty 2^{-c_{n-1} / p} f_n
\]
the closed linear span of $F$ includes $E$.  Thus, $F$ is a generating set for $\ell^p$.  Note that $\norm{f_0} = 1$.

Suppose $\alpha_0, \ldots, \alpha_M$ are rational points.  When $1 \leq j \leq M$, set
\[
\mathcal{E}_j = |\alpha_0 2^{-c_{j-1} / p} + \alpha_j |^p - |\alpha_0|^p 2^{-c_{j-1}}.
\]
It follows that 
\begin{eqnarray*}
\norm{\alpha_0 f_0 + \ldots +\alpha_M f_m}^p & = & |\alpha_0|^p \norm{f_0}^p + \mathcal{E}_1 + \ldots + \mathcal{E}_M\\
& = & |\alpha_0|^p + \mathcal{E}_1 + \ldots + \mathcal{E}_m.
\end{eqnarray*}
Since $\mathcal{E}_1$, $\ldots$, $\mathcal{E}_M$ can be computed from $\alpha_0, \ldots, \alpha_M$, $\norm{\alpha_0 f_0 + \ldots + \alpha_M f_M}$ can be computed from $\alpha_0, \ldots, \alpha_M$.  Thus, $F$ is an effective generating set.
\end{proof}

\begin{lemma}\label{lm:X.computes.C.0}
Every oracle that with respect to $F$ computes a
unimodular scalar multiple of $e_0$ must also compute $C$.
\end{lemma}

\begin{proof}
Suppose that with respect to $F$, $X$ computes a vector of the form $\lambda e_0$ where $|\lambda| = 1$.  It suffices to show that $X$ computes $(1 - \gamma)^{-1/p}$.  

Fix a rational number $q_0$ so that $(1 - \gamma)^{-1/p} \leq q_0$.  Let $k \in \N$ be given as input.  Compute $k'$ so that $2^{-k'} \leq q_0 2^{-k}$.  Since $X$ computes $\lambda e_0$ with respect to $F$, we can use oracle $X$ to compute rational points $\alpha_0, \ldots, \alpha_M$ so that 
\begin{equation}
\norm{\lambda e_0 - \sum_{j = 0}^M \alpha_j f_j} < 2^{-k'}.\label{inq:1}
\end{equation}
We claim that $|(1 - \gamma)^{-1/p} - |\alpha_0| | < 2^{-k}$.  For, it follows from (\ref{inq:1}) that $|\lambda - \alpha_0 (1 - \gamma)^{1/p}| < 2^{-k'}$.  Thus, $|1 - |\alpha_0| (1 - \gamma)^{1/p}| < 2^{-k'}$.  Hence, 
\[
|(1 - \gamma)^{-1/p} - |\alpha_0|| < 2^{-k'}(1 - \gamma)^{-1/p} \leq 2^{-k'}q_0  \leq 2^{-k}.
\]
Since $X$ computes $\alpha_0$ from $k$, $X$ computes $(1 - \gamma)^{-1/p}$.
\end{proof}

\begin{lemma}\label{lm:X.computes.C.1}
If $X$ computes a surjective isometric endomorphism of $\ell^p$ with respect to $(E,F)$, then $X$ must also compute $C$. 
\end{lemma}

\begin{proof}
Let $T$ be a surjective endomorphism of $\ell^p$, and suppose $X$ computes $T$ with respect to $(E,F)$.   
By Theorem \ref{thm:classification}, there exists $j_0, \lambda$ so that $T(e_{j_0}) = \lambda e_0$ and $|\lambda| = 1$.  
So, by Lemma \ref{lm:X.computes.C.0}, $X$ computes $C$. 
\end{proof}

\begin{lemma}\label{lm:C.computes.e_0}
With respect to $F$, $C$ computes $e_0$.
\end{lemma}

\begin{proof}
Fix an integer $M$ so that $(1 - \gamma)^{-1/p} < M$.  

Let $k \in \N$.  Using oracle $C$, we can compute an integer $N_1$ so that $N_1 \geq 3$ and 
\[
\norm{ \sum_{n = N_1}^\infty 2^{-c_{n - 1}/p} e_n } \leq \frac{2^{-(kp + 1)/p}}{2^{-(kp + 1)/p} + M}.
\]
We can use oracle $C$ to compute a rational number $q_1$ so that $|q_1 - (1 - \gamma)^{-1/p}| \leq 2^{-(kp + 1)/p}$.  Set
\[
g = q_1 \left[ f_0 - \sum_{n = 1}^{N_1 - 1} 2^{-c_{n-1}/p} f_n \right].
\]
It suffices to show that $\norm{e_0 - g} < 2^{-k}$.  Note that since $1 - \gamma < 1$,\\ $|q_1(1 - \gamma)^{1/p} - 1| \leq 2^{-(kp + 1)/p}$.  Note also that $|q_1| < M + 2^{-(kp +1)/p}$.   Thus, 
\begin{eqnarray*}
\norm{e_0 - g}^p & = & \norm{e_0 - q_1(1 - \gamma)^{1/p} e_0 - q_1 \sum_{n = N_1}^\infty 2^{-c_{n - 1/p}}e_n}^p\\
& \leq & |q_1 (1 - \gamma)^{1/p} - 1|^p + |q_1|^p \norm{\sum_{n = N_1}^\infty 2^{-c_{n-1}/p} e_n}^p\\
& < & 2^{-(kp + 1)} + 2^{-(kp + 1)} = 2^{-kp}
\end{eqnarray*}
Thus, $\norm{e_0 - g} < 2^{-k}$.  This completes the proof of the lemma.
\end{proof}

\begin{lemma}\label{lm:C.computes.identity}
With respect to $(E,F)$, $C$ computes a surjective linear isometry of $\ell^p$.
\end{lemma}

\begin{proof}
By Lemma \ref{lm:C.computes.e_0}, $C$ computes $e_0$ with respect to $F$.  Thus, 
$C$ computes $\{e_n\}_{n = 0}^\infty$ with respect to $F$, and it follows that 
$C$ computes the identity map with respect to $(E, F)$. 
\end{proof}

\section{Additional results}\label{sec:additional}

Suppose $n$ is a positive integer and $1 \leq p < \infty$.  Define $\ell^p_n$ to be the 
set of all $f \in \ell^p_n$ so that $f(j) = 0$ whenever $j \geq n$; i.e. $\supp(f) \subseteq \{0, \ldots, n-1\}$.  
Thus, $\ell^p_n$ is a subspace of $\ell^p$, and $\{e_0, \ldots, e_{n-1}\}$ is an effective generating set for 
$\ell^p_n$.  Now, suppose $p$ is computable and $p \neq 2$.  Let $F$ be an effective generating set for $\ell^p_n$.  Via the methods of the previous section, we can show that there are disjointly supported unit vectors $f_1, \ldots, f_n \in \ell^p_n$ so that each $f_j$ is computable with respect to $F$.  Thus, 
$f_1, \ldots, f_n$ generate $\ell^p_n$.  It then follows that $\ell^p_n$ is computably categorical.  However, since $p \neq 2$, $\ell^p_n$ is not a Hilbert space.   Thus, \it there is a computably categorical Banach space that is 
not a Hilbert space.\rm 

\section{Conclusion}\label{sec:conclusion}

To summarize, we have investigated the complexity of isometries between computable copies of $\ell^p$.  We have shown that the halting set bounds the complexity of computing these isometries and that this bound is optimal.  Along the way we have strengthened an important inequality due to J. Lamperti.  
These results stand as a contribution to the emergent program of grafting computable structure theory onto computable analysis.  It is anticipated that there will be many other interesting discoveries in this area and that the proofs will present opportunities to blend methods from classical analysis and computability theory.

\section*{Acknowledgements}

I thank Joe Cima, Johanna Franklin, Xiang Huang, and Don Stull for helpful and inspiring conversation.  I also thank the anonymous referees for frankly sharing many helpful suggestions for improving the style of the paper.
This work was supported by a Simons Foundation Collaboration Grant for Math\-e\-ma\-ti\-cians.


\def\cprime{$'$}
\providecommand{\bysame}{\leavevmode\hbox to3em{\hrulefill}\thinspace}
\providecommand{\MR}{\relax\ifhmode\unskip\space\fi MR }
\providecommand{\MRhref}[2]{%
  \href{http://www.ams.org/mathscinet-getitem?mr=#1}{#2}
}
\providecommand{\href}[2]{#2}

\end{document}